\documentclass{amsart}
\usepackage[margin=1.5in]{geometry}

\usepackage[all]{xy}
\usepackage{tikz-cd}
\usepackage{enumerate}
\usepackage{amsfonts}
\usepackage{amssymb, amscd}
\usepackage{graphicx}
\usepackage{mathtools}
\usepackage{xcolor}
\usepackage{float}

\newcommand{\R}{\mathbb{R}}

\newcommand{\pf}{\n{\em Proof.}   }

\newcommand{\n}{\noindent}

\newcommand{\sn}{\smallskip\noindent}
\newcommand{\ms}{\medskip}
\newcommand{\bn}{\bigskip\noindent}

\newtheorem{teo}{Theorem}[section]
\newtheorem{cor}[teo]{Corollary}

\newtheorem{lema}[teo]{Lemma}

\theoremstyle{definition}

\newtheorem{conj}[teo]{Conjecture}

\title{Convex bodies all whose sections (projections) are equal}

\author[L. Montejano]{Luis Montejano }
\address{IMATE, UNAM, Sede Juriquilla, Queretaro, Mexico}
\email{luis@im.unam.mx}
\date{\today}

\begin{document}
\maketitle
\tableofcontents
\section{Introduction}

The purpose of this paper is to answer the following question: 

\smallskip

\emph{If all hyperplane sections through the origin of a convex body are "equal", 
is the convex body  "equal" to the ball? }

\smallskip

The meaning of the notion "equal" will change in the course of this paper.

\smallskip

Similarly, we are interested in the following problem:

\smallskip

\emph{If all orthogonal projections of a convex body onto hyperplanes are "equal", 
is the convex body  "equal" to the ball? }

\smallskip

We believe that topology and convex geometry are deeply and beautifully interrelated in the solution and understanding of these problems.

\smallskip

During this paper, unless otherwise stated,  $B$ is always an $(n+1)$-dimensional convex body with the origin as an interior point and $n\geq 2$.

\section{Sections with the same area}

\medskip

The first meaning of "equal" is same "area".

\medskip 

\emph{ If all hyperplane sections through the origin of a convex body $B$ have equal 
$n$-dimensional volume, does $B$ have the $(n+1)$-dimensional volume of the corresponding ball?}

\smallskip

The answer to this question is by far a resounding no. The reader will find it easy to construct a counterexample. However, if we add the symmetry hypothesis to the question, the answer becomes yes.   More precisely: 

\smallskip

\begin{teo}\label{teoa}
If all hyperplane sections through the origin of a symmetric convex body $B$ have equal 
$n$-dimensional volume, then  the convex body $B$ is a ball centered at the origin.
\end{teo}

\pf The proof of this theorem uses analysis.  We give here a sketch of the proof using harmonic integration. See Falconer's paper  \cite {F} or Schneider`s Book \cite{Sc2}.

First of all, let  $f: \mathbb{S}^n  \rightarrow  \mathbb{R}$ be a continuous function such that
\smallskip
\[
 \int_{\langle x,y \rangle = 0}f(x) dx = 0, \qquad \mbox{for every }y \in \mathbb{S}^n\,,
\]
where integration  refers to the usual measure in the $(n-1)$-sphere.
Then the classical theorem of Funk-Hecke on spherical harmonics (see \cite{G}) implies that $f$ is an odd function; that is, $-f(x) = f(-x)$,  for almost every  $x \in \mathbb{S}^n$.

Let now  $B_1$ and $B_2$ be two $(n+1)$-dimensional convex bodies that are symmetric with center at the origin and assume that the corresponding parallel $n$-dimensional areas of their  sections through the origin are equal. We shall show that 
$B_1=B_2$.  For this purpose, let  $f_1, f_2\colon \mathbb{S}^{n}  \rightarrow \mathbb{R}$ be the radial functions of $B_1$ and $B_2$. 
Note that because $B_1$ and $B_2$ are symmetric, $f_1$ and $f_2$ are even functions. 
Moreover, by hypothesis 
\smallskip
\[
\frac{1}{n} \int_{\langle x,y \rangle=0} f_1(x)^{n}  dx = \frac{1}{n}
\int_{\langle x,y \rangle=0} f_2(x)^{n}  dx,
\]
for every $y \in \mathbb{S}^{n-1}$.

\smallskip

By our first argument,  $f^{n}_1 - f^{n}_2$ is an odd function,
but since $f_1$ and $f_2$ are even functions, $f^{n}_1 = f^{n}_2$.
Moreover, since $f_1 \geq 0$  and $f_2 \geq 0$, we obtain that $f_1 = f_2$ and hence that $B_1=B_2$.

Suppose now $B$ is a symmetric convex body with the property that all its hyperplane sections through the origin have equal 
$n$-dimensional volume, and  let $G \in SO_n$ be  a linear isometry.
Then, by the above $B = GB$,  for every $G \in SO_n$, and consequently $B$ is a ball centered at the origin.
\qed

\section{Congruent and similar sections} 

The second meaning of "equal" is congruence.
\smallskip

\begin{teo}[Schneider's Theorem]\label{teosch}
If all hyperplane sections through the origin of a  convex body $B$ are congruent, 
 then the convex body $B$ is an $(n+1)$-ball centered at the origin.
\end{teo}

This time, the hypothesis of symmetry is not necessary.  The theorem was proved by S\"uss \cite{Su} for $n=2$. In 1970, using  topological ideas, Mani \cite{Ma} proved it for n=even and in 1979, Burton \cite{B} proved it for $n=3$. Finally, Rolf Schneider \cite{Sch} in 1980, using analysis, proves it in general.  In 1990, using the topological ideas of Hadwiger and Gromov, Montejano \cite{M1} proved the following result which, together with the false center theorem, allows an alternative proof of Schneider's theorem to be given.

\smallskip

\begin{teo}\label{teom}
If all hyperplane sections through the origin of a convex body $B$ are affinely equivalent,    
 then every hyperplane section of $B$ is centrally symmetric.
\end{teo}

The proof  of Theorem \ref{teom} uses topological ideas. Indeed, it uses the notion of field of convex bodies introduced by Hadwiger and developed by Mani in \cite{Ma}. 

\smallskip

\subsection{Fields of convex bodies}

Let $\bf K^n$ be the space of all compact convex sets in $\mathbb R^n$ with the Hausdorff metric topology.

A \emph{field of convex bodies tangent to} $\mathbb S^n$ is a continuous function
$$\kappa: \mathbb S^n\to \bf K^{n+1},$$
such that $\kappa(u)\subset u+u^\perp\subset \mathbb R^{n+1},$ for every $u\in\mathbb S^n$, where $u^\perp$ denotes the subspace of $R^{n+1}$ orthogonal to $u$. 

If in addition, $\kappa(u)$
 is congruent (affinely equivalent) to the convex body $K\subset \mathbb R^n$, for every $u\in\mathbb S^n$, then we obtain a \emph{field of convex bodies tangent to} $\mathbb S^n$ and congruent (affinely equivalent) to $K$. If in addition, $\kappa(u)-u=\kappa(-u)+u$, then we have a \emph{complete turning} of $K$ in  $\mathbb R^{n+1}.$

If all hyperplane sections through the origin of a convex body $B$ are congruent (affinely equivalent), then there is a field  convex bodies tangent to $\mathbb S^n$ and congruent (affinely equivalent) to $\mathbb R^{n}\cap B$ 
$$\kappa: \mathbb S^n\to \bf K^{n+1},$$
defined as follows: 
$$\kappa(u)=u+(u^\perp\cap B), \quad \mbox { for every } u\in\mathbb S^n.$$
Obviously, this field is a complete turning because $\kappa(u)-u=\kappa(-u)+u$, for every  $u\in \mathbb S^n$.

\smallskip

Note that given a field of convex bodies $\kappa: \mathbb S^n\to \bf K^{n+1}$ tangent to $\mathbb S^n$ and congruent to $K$, we may always assume without loss of generality that, for every $u\in\mathbb S^n$, the circumcenter of $\kappa(u)$ is the point $u\in (u+u^\perp)$.

\bigskip 

The link between Hadwiger's notion of field of convex bodies and the topology of Lie groups traces back to the work of Steenrod \cite{St} and Gromov \cite{G}. Every vector bundle $\xi:E\to\mathbb S^n$ with base the sphere $\mathbb S^n$, fiber $\R^m$ and structure group $GL(m,\R)$, can be obtained from 
$\mathbb B^n_+ \times \R^m$ disjoint union $\mathbb B^n_- \times \R^m$ by gluing the first copy $\mathbb S^{n-1}\times \R^m\subset \mathbb B^n_+ \times \R^m$ with the second copy 
$\mathbb S^{n-1}\times \R^m \subset \mathbb B^n_- \times \R^m$ via a fiber preserving homeomorphism 
\bigskip
$$S^{n-1}\times \R^m\to S^{n-1}\times \R^m$$
that glue every fiber $\{x\}\times \R^m$ with the fiber $\{x\}\times \R^m$ using an element $g_x \in GL(m,\R)$,
where  $\mathbb B^n_+$ and  $\mathbb B^n_-$ are respectively, the north and south hemisphere of $\mathbb S^n$.  \quad The map
$g:\mathbb S^{n-1}\to GL(m,\R)$ given by $g(x)=g_x$, is called the \emph{characteristic map} of the vector bundle $\xi$.
It is not difficult to see that two vector bundles are equivalent if and only if their corresponding characteristic maps are homotopic. 

The existence of a field of convex bodies tangent to $\mathbb S^n$ and congruent to $K$ implies that the tangent bundle 
T$\mathbb S^n$ can be obtained gluing the copies $\mathbb B^n_+ \times \R^m$ and $\mathbb B^n_- \times \R^m$ using only isometries  that fixes $K$.  In other words:

\medskip

\noindent \emph{ There exist a field of convex bodies tangent to $\mathbb S^n$ and congruent to $K$ if and only if 
the characteristic map }
\bigskip
$$
\begin{tikzcd}[column sep=small]
\mathbb S^{n-1}\arrow[dr,"f",swap] \arrow[rr, "\chi_n"] & &SO_n \\
&G_K\arrow[ur,"i",swap]&
\end{tikzcd}
$$
\emph{factorizes through } $G_K=\{g\in SO_n\mid g(K)=K\}$.
 If this is so, then we say that the structure group of T$\mathbb S^n$ \emph{reduces to } $G_K$. 

\bigskip

The main idea in the proof of Theorem \ref{teom}, is that a complete turning of $K$ is only possible if $K$ has a center of symmetry (indeed, if $n=3,7$, the fact that the tangent bundle of $\mathbb S^3$ and $\mathbb S^7$ is parallelizable implies that a complete turning of $K$ is possible if and only if $K$ has a center of symmetry).

\bigskip

Since vector bundles over contractible spaces are trivial, we are going to take advantage of the existence of the field of convex bodies 
$\kappa: \mathbb S^n\to \bf K^{n+1},$ tangent to the sphere $\mathbb S^n$ and congruent to $K$, to construct a continuous map
$$\Phi: \mathbb B^n_+\to SO_n,$$
such that $\Phi(x)(K)=\kappa(x)$, for every  $x\in \mathbb B^n_+.$

\medskip

Suppose $K$ is not symmetric. We may assume without generality that there is a point $x_0$ in the boundary of $K$ such that $-x_0\notin K$
and hence for every $g\in SO(n)$, $g(x_0)\not=-x_0$.

Note that 
$$\{\Phi(x)(x_0)\}$$
is a field of vectors tangent to $\mathbb B^n_+$.  Furthermore, for every $u\in \mathbb S^{n-1}$, we have that $\Phi(u)(x_0)\not=-\Phi(-u)(x_0).$
We are going to add a small annulus to $B^n_+$ at the boundary, to obtain a larger $n$-dimensional ball $\tilde B^n$ and we are going to take advantage of this annulus  to define on it a tangent vector field that coincides with the one we have in $B^n$ and with an additional property.
The idea is that for every point $u \in \mathbb S^{n-1}$, we will use the annulus to rotate from the vector $\Phi(u)(x_0)$ towards the vector $\Phi(-u)(x_0)$. Since $\Phi(u)(x_0)\not=-\Phi(-u)(x_0),$ we can do this unambiguously, in such a way that at the end on the border of $\tilde B^n$, the tangent vector at the point $u\in\partial\tilde B^n$, coincides with the tangent vector at the point $-u\in\partial \tilde B^n$.  Using this procedure, we obtain a complete turning of a nonzero vector field in the sphere $\mathbb S^{n}$, which is a contradiction to the well known result that there is not a section to the canonical vector bundle of $n$-subspaces in $\mathbb R^{n+1}.$  See \cite{St}.

Suppose $K_1, K_2$ are convex bodies who have as ellipsoid of minimal volume containing them the unit ball.  It is easy to see that if $K_1$ and $K_2$ are affinely equivalent, then they are actually congruent.  Suppose, now $K\subset \mathbb R^n$ is a convex body with the unit ball as a ellipsoid of minimal volume containing it, and let $\kappa: \mathbb S^n\to \bf K^{n+1}$
 be a field of convex bodies tangent to $\mathbb S^n$ and affinely equivalent to the convex body $K\subset \mathbb R^n$, then there is a field of convex bodies tangent to $\mathbb S^n$ congruent to $K$.  For every $x\in \mathbb S^n$, let $E_x\subset x+x^\perp$ be the ellipsoid of minimal volume containing $\kappa(x)$ and let $h_x$ be the affine map that translates and dilates the principal axes of $E_x$ to obtain the unit ball. It is easy to observe that the affine map $h_x$ varies continuously with $x$.  Hence $\kappa': \mathbb S^n\to \bf K^{n+1}$, given by $\kappa'(x)=h_x(\kappa(x))$, is a   field of convex bodies tangent to $\mathbb S^n$ congruent to $K$.  By all the above, \emph{if $\kappa: \mathbb S^n\to \bf K^{n+1}$ field of convex bodies tangent to $\mathbb S^n$ and affinely equivalent, then for every $x\in \mathbb S^n$, $\kappa(x)$ is symmetric.}
 
 \bigskip
 
\subsection{The proof of Schneider's theorem and similar sections} 
Summarizing, Theorem  \ref{teom} is true because a complete turning of $K$ is only possible if $K$ has a center of symmetry. 
This result, in combination with the beautiful Larman's False Center theorem \cite{L}, give rise to a topological proof of Schneider`s theorem.

\begin{teo}[Larman`s False Center theorem]\label{teol}
If all hyperplane sections through the origin of a convex body $B$ have a center of symmetry,   
 then either $B$ is an ellipsoid or $B$ is symmetric with respect the origin.
\end{teo}

The  proof of Schneider's theorem \ref{teosch} goes as follows. If all hyperplane sections through the origin of $K$ are congruent, by Theorem \ref{teom},  
 then every hyperplane section through the origin is centrally symmetric. By Larman's False Center Theorem \ref{teol}, either $B$ is symmetric with center the origin and Alexandrov`s Theorem \ref{teoa} implies that $B$ is a ball centered at the origin, or $B$ is an ellipsoid in which case it  is easy to see directly that $B$ is again a ball centered at the origin.
 
 \smallskip 
 
 The third meaning of equal is similarity. If $B$ is a $(n+1)$-ball with the origin as an interior point but not necessarily centered  at the origin, then all hyperplane sections of $B$ through the origin are $n$-balls and hence all are similar. Our next theorem states that this is always the case.
 
 \begin{teo}[Montejano]\label{teom1}
If all hyperplane sections through the origin of a convex body $B$ are similar, 
 then the convex body $B$ is an $n$-ball not necessarily centered at the origin.
\end{teo}

A sketch of the proof is the following. Since similarities are affine equivalences, by Theorem \ref{teom}, all hyperplane sections of $B$ through the origin have a center of symmetry.  By Larman`s False Center Theorem \ref{teol}, either the origin is the center of symmetry of $B$ or $B$ is an ellipsoid.  Using a topological argument it is possible to prove that in the first case all hyperplane sections of $B$ through the origin are 
not only similar but actually 
congruent and  hence by Schneider's theorem \ref{teosch}, $B$ is a ball or in the second case, if $B$ is an ellipsoid, it is easy to directly verify that our hypothesis implies that $B$ it is actually a ball.

\bigskip
\bigskip

  \section{Affinely equivalent sections and the Banach conjecture}

The forth meaning of equal is affinely equivalence. 

\begin{conj}\label{conj1}
If all hyperplane sections through the origin of a  convex body $B$ are affinely equivalent, 
 then the convex body $B$ is an ellipsoid.
\end{conj}

\bigskip 

It turns out that Conjecture \ref{conj1} is equivalent to the Banach Conjecture over the reals.

\bigskip

\subsection{The Banach conjecture} In 1932, in his book \cite{Ba},
Stephan Banach asked   the following question: 

\sn

{\em Let $V$ be a Banach space, real or complex, finite or infinite dimensional, all of whose $n$-dimensional subspaces, for some fixed integer $n$, $2\leq n<\dim(V)$, are isometric to each other. Is it true that $V$ is a Hilbert space?}

\ms

This conjecture  was proved first  for $n=2$  and  real $V$   in 1935 by Auerbach, Mazur and Ulam  \cite{AMU} and in 1959 for all $n\geq 2$ and infinite dimensional real $V$ by A.~Dvoretzky \cite{Dv}. In 1967, M.~Gromov \cite{G} proved  the conjecture for  even $n$ and all $V$, real or complex,  for odd $n$  and  real $V$  with $\dim(V)\geq n+2$, and for odd $n$ and  complex $V$ with $\dim(V)\geq 2n$.  V.~Milman \cite{Mi} extended Dvoretzky's theorem to the complex case, in particular reproving Banach's conjecture for infinite dimensional complex space $V$.
 Recently, in 2021,  Bor, Hern\'andez-Lamoneda, Jim\'enez-Desantiago and Montejano \cite{BHJM} proved the Banach Conjecture if $V$ is real and $n\equiv 1$ mod $4$, with the possible exemption of $n=133$ and a little later,  Bracho and Montejano \cite{BM} proved the Banach Conjecture if $V$ is complex and $n\equiv 1$ mod $4$. A thorough account of the history of this conjecture is found  in the notes on \S9 in \cite[p.~206]{MMO}. 
We also recommend \cite{Pe}.

\medskip

Our next purpose is to prove that the Banach conjecture over the reals is equivalent to Conjecture \ref{conj1}.
First note that Banach's conjecture is a codimension one problem: since every Banach space, all of whose subspaces of a fixed dimension $n\geq2$ are Hilbert spaces, is itself a Hilbert space, which easily follows from the elementary characterization of a norm coming from an inner product via the ``parallelogram law", an affirmative answer for $n$ in codimension one implies immediately an affirmative answer for $n$ in all codimensions. 

Note next that two Banach spaces $V_1$ and $V_2$ are isometric if there is a linear isomorphism
$f:V_1\to V_2$ that preserves the norm. That is, two Banach spaces $V_1$ and $V_2$ are isometric if their unit balls are linearly equivalent.
To conclude, note that a finite dimensional Banach space $V$ is a Hilbert space if and only if $V$ is isometric to the Euclidean space, that is,
if and only if its unit ball is an ellipsoid.

Finally, in the solution of Conjecture \ref{conj1}, we may always assume that not only $B$ but all hyperplane sections of $B$ through the origin have as a center of symmetry the origin.
This is so, because by Theorem \ref{teom} every section of $B$ has a center of symmetry and therefore by Larman`s False Center Theorem \ref{teol} 
either $B$ is an ellipsoid or the origin is the center of $B$. 
\bigskip

\subsection{Topology of Lie groups} From now on, until the end of this section, suppose $B$ is a convex body with the property that all its hyperplane sections through the origin  are affinely equivalent. Our first interest is to answer the following question: 

\smallskip

\centerline{\emph{ What can we say about the sections of $B$?}}

\smallskip

For example, due to the Theorem \ref{teom}, we know that all these sections have a center of symmetry 
but, do these sections share some other property?

  Choose a convex set $K\subset \mathbb R^n$ affinely equivalent to all hyperplane sections of $B$ through the origin with the additional property that 
  the ellipsoid of minimal volume containing $K$ is the unit $n$-ball.  Define $G$ as the group of symmetries of $K$, that is, $G$ is the  subgroup of linear isomorphism in $GL(n, \mathbb R)$  keeping fixed $K$ and with positive determinant.  Note that every element of $G$ fixes also the unit $n$-ball that therefore $G\subset SO_{n}$.  As we shall see, $G$ is a compact Lie group relevant in the solution of our previous question. 
  
  As in the sketch of the proof of Theorem \ref{teom}, in Section 3,  there is a field of convex bodies tangent to $\mathbb S^n$ and affinely equivalent to $K$. This implies that the structure group of the tangent bundle of the sphere $\mathbb S^n$ can be reduced to $G$ or in other words that the characteristic map of $T\mathbb S^n$ 
  $$\chi_n:\mathbb S^{n-1}\to SO_n$$ 
  can be factorized  through $G$. See Section  of Steenrod's Book \cite{St}  or Mani's paper \cite{Ma}.
  
  If $n$ is even and $G$ is not transitive, the structure group of the tangent bundle of the sphere $\mathbb S^n$ can not be reduced to $G$.  
  This is so because if  there is a map $f:\mathbb{S}^{n-1}\to SO_n$ homotopic to $\chi_n$, such that $f(\mathbb{S}^{n-1})\subset G$ 
and $e:SO_n\to \mathbb S^n$ is the evaluation map,  then $ef$ is homotopic to $e\chi_n$. The non transitivity of $G$ implies that 
$e:SO_n\to \mathbb S^n$ is not onto 
and therefore that $ef$ is null homotopic. Thus, 
$e\chi_n$ is null homotopic, which is a contradiction in even dimensions, where we can easily calculate the even degree of
 $e\chi_n$.  Consequently, if $n$ is even,  a field of convex bodies tangent to $\mathbb{S}^{n}$ affinely equivalent to $K$ implies that $G$ is transitive and consequently that $K$ is an $n$-ball. In contrast, for $n=3$,
there is a field of convex bodies tangent to $\mathbb{S}^{n}$ and congruent to $K$, for every convex body $K\subset \mathbb{R}^{n}$, because 
$\mathbb S^3$ is parallelizable. 

Summarizing, if $n$ is even, the answer to our question; what can we say about the sections of $B$? is that all these sections are affinely equivalent to a ball and hence all of them are ellipsoids. This immediately implies that $B$ is an ellipsoid, solving conjecture 1 when $n$ is even and the Banach Conjecture when $n$ is even and $V$ is a Banach space over the reals. 

The case  $n=odd$ is more complicated.  First note that if $n=3,7$, this topological technique does not give us information about the sections of $B$, because $\mathbb S^3$ and $\mathbb S^7$ are parallelizable. We shall prove next that if $n\equiv 1$ mod $4$, with the possible exemption of $n=133$, a field of convex bodies tangent to $\mathbb{S}^{n}$ affinely equivalent to $K$ implies that $K$ is an affine body of revolution.  

Suppose the characteristic map of the sphere $\chi_n$ factorize through the maximal connected subgroup 
$G\subset SO_n$, that is, 

$$\mathbb S^{n-1} \to G \hookrightarrow SO_n.$$
We have two cases:
\begin{enumerate}
\item $G$ is a irreducible representation, that is, the action of $G$ does not fix any proper subspace, and 
\item the action of $G$ fixes a proper subspace $\Gamma^k$; $1\leq k\leq n-1$.
\end{enumerate}

In the first case, mathematicians have extensively studied irreducible representations, in particular those for which the structural group of the space tangent to the sphere can be reduced to them.  In particular, Leonard \cite{L} proved that if $G\subset SO_n$ is a maximal connected irreducible representation and the 
characteristic map of the sphere $\chi_n$ factorize through $G$, then $G$ is a simple group. 

If this is so, we have several options: 

\begin{itemize}
\item $G$ is a classical group; $SO_k, SU_k, Sp_k$,
\item $G$ is an Spin group; $Spin_k$,
\item $G$ is one of the exceptional Lie groups, $G_2, F_4, E_6, E_7$ or $E_8$.
\end{itemize}

Furthermore, in 2006, Cadek and Crabb proved that under the same hypothesis for $G$, if $n\geq 8$,  then $G$ is not isomorphic to 
$SO_k, SU_m, Sp_m$, with $k\geq 4, m\geq 2$.  If $n\equiv 1$ mod $4$, this rules out the classical groups, with the exemption of $n=5$.  We leave this exceptional case for the next section. Furthermore, it can be proved that every irreducible representation of $Spin_k$, which does not factor through $SO_m$, is even dimensional. In our case, it is clear that $G$ does not factor through $SO_m$, so if $n$ is odd, we can rule out the possibility of an Spin group for $G$. 

Suppose now, $n\equiv 1$ mod $4$. If this is the case, dim($G$) is not two small with respect to $n$ and hence $G$ is not an exceptional Lie group,
with the possible exception of the Lie group $E_7\subset O_{133}$. This is so, because it can be proved that in this case, dim($G)\geq 2n-3$ (see Proposition 3.1 of \cite{CC}). Hence to rule out the exceptional groups one can simply check (e.g., in Wikipedia) the following table in which we list the smallest  irreducible 
representation for them, and the smallest  irreducible representation congruent to 1 mod 4 is in boldface, verifying that in all the cases, with the exception of $E_7$, dim($G)\leq 2n-4$.

\bigskip
\bigskip

\begin{center}
\begin{tabular}{|l|c|c|c|c|c|}
 Group&  $G_2$&$F_4$  &$E_6$  &$E_7$& $E_8$  \\\hline
 $\dim G$&14  &52  &78  &133 & 248  \\\hline
\mbox{Irreps}& 7&26  &27  & 56 & 248 \\
 &14&52  & 78 &\fbox{{\color{red} 133}}  & 3875\\
 &27&{\color{red} 273}&351&912&\vdots \\
 &64& \vdots & {\color{red} 2925} & \vdots & {\color{red} 1763125}\\
 &{\color{red} 77}&\vdots &\vdots &\vdots &\vdots
\end{tabular}
\end{center}

\bn
\bigskip

All the above implies that if $G$ is irreducible and $n\equiv 1$ mod $4$, then $G$ is $E_7$ or is conjugate to $O_n$. Consequently, in this last case, $K$ must be a ball, all the sections must be ellipsoids and $B$ must be an ellipsoid, as we wished.

\bigskip

The second case is when the action of $G$ fixes a proper subspace $\Gamma^k$; $1\leq k\leq n-1$.

If $n=4k+1$,  the tangent space of the sphere $T\mathbb S^n$ splits 
$$T\mathbb S^n= e^1\oplus \eta^{4k},$$
where $e^1$ is a vector bundle of dimension 1 and  $\eta^{4k}$ is unsplittable.

\bigskip 

From here we deduce that $\Gamma^k$ is either $1$ or $(n-1)$-dimensional, and 
$G$ is a subset of a conjugate copy of $SO_{n-1}$. Furthermore, using  argument very similar to the argument used in the proof that 
$G=SO_n$, when $n$ is even (or see the case $n=5$), it is possible to prove that $G$ is actually a conjugate copy of $SO_{n-1}$. 
This give rise to the case in which $K$ is a body of revolution. 

\medskip

Summarizing, \emph{suppose $B$ is an $(n+1)$-dimensional  convex body with the property that all its hyperplane sections through the origin  are affinely equivalent, $n\equiv 1$ mod $4$, $n\not= 5, 133$.  Then, every hyperplane section of $B$ through the origin is an affine body of revolution.}

\subsection{The case $n=5$}
This case is an exceptional case in our proof of the Banach conjecture but it is also interesting enough to illustrate the true complexity of the conjecture. This section will be dedicated to its complete proof. 

Let $B$, $K$, and $G$ defined as in the previous section but this time $B$ is a symmetric convex body in $\mathbb R^6$, and $G=\{g\in SO_5\mid
g(K)=K\}$ is a compact Lie subgroup of $SO_5$. Furthermore, we know that the characteristic map of the tangent space of $\mathbb S^5$

$$
\begin{tikzcd}[column sep=small]
\mathbb S^{4}\arrow[dr,"f",swap] \arrow[rr, "\chi_5"] & &SO_5 \\
&G\arrow[ur,"i",swap]&
\end{tikzcd}
$$
factorizes through $G$. 

\bigskip

\emph{Suppose first $G$ leaves invariant a proper subspace of $\mathbb{R}^{5}$. We shall prove that in this case that $K$ is a body of revolution.}

\medskip
 By hypothesis, there is a $k$-dimensional  subspace $\Lambda$ invariant under  $G$. This immediately implies that there is a continuous field of $k$-planes in $\mathbb{S}^{5}$.  By Theorem 27.18 of \cite{St},  we know that $\mathbb{S}^{5}$ admits a continuous field of $k$-planes if and only if  $k=1$ or $k=4$. So, assume without loss of generality that $k=1$, and therefore that $\Lambda$ is a line invariant under $G$. Suppose without loss of generality that $\Lambda$ is the line through the origin orthogonal to $\mathbb{R}^4$, in such a way that $G\subset SO_4$.  We will prove that $G$ acts transitively on $\mathbb R^4$, thus proving that $K$ is a body of revolution. 
 
 Given any $5$-dimensional plane in $\mathbb R^6$, it is easy to prove that there is a unique complex plane contained in it. It is for this reason that there is a field of complex planes tangent to $\mathbb{S}^{5}$. This implies that the structural group of T$\mathbb{S}^{5}$ can be reduced to $SU_2$.  Thus, we may assume that $\chi_5:\mathbb S^4\to SU_2$ is the characteristic map of T$\mathbb{S}^{5}$. If $e: SO_4\to \mathbb{S}^{3}$ is the evaluation, hence, $e\chi_5:\mathbb S^4\to S^3$ is not null homotopic. To see this, note that  $SU_2$ is homeomorphic to  $\mathbb{S}^{3}$ and 
the evaluation $e:SU_2\to \mathbb S^3$ is a homeomorphism. Therefore, if $e\chi_5: \mathbb{S}^{4}\to \mathbb{S}^{3}$ is homotopically trivial,  then the same holds for $\chi_5: \mathbb{S}^{4}\to SU_2$, but this implies that the characteristic map  of T$\mathbb{S}^{5}$ is homotopic to a constant, and therefore that  T$\mathbb{S}^{5}$ is parallelizable which is a contradiction.  
 
We know that the structural group of T$\mathbb{S}^{5}$ can be reduced to $G$. Therefore, the characteristic map $\chi_5:\mathbb S^4\to SU_2$   is homotopic on $SO_4$ to a map $f:\mathbb{S}^{4}\to G$. 
This implies that  $e\chi_5, ef:\mathbb S^4\to \mathbb S^3$ are homotopic.  
If $G$ does not act transitively on 
$\mathbb R^4$, hence $ef$ is null homotopic, but this is a contradiction to the fact that $e\chi_5$ is not null homotopic.  Consequently, 
 $G$ acts transitively on $\mathbb R^4$ and $K$ is a body of revolution, as we wished.

 \medskip

Suppose now that $G\subset SO_5$ does not leave invariant a proper subspace of $\mathbb{R}^{6}$.  That is, we must study the 
\emph{irreducible representations on $\mathbb R5$.}

Consider $S$ the collection of $3\times3$ real symmetric matrices with zero  trace. Then, $S$ is a real vector space of dimension $5$ with the following natural interior product: given $A, B \in S$,  
$$A\odot B=tr(AB).$$

The group $G=SO_3$ defines the following representation:
$g(A)=gAg^{-1}=gAg^t,$
for every $g\in G$ and $A\in S$.  

Clearly, $G$ acts linearly on $S$ and furthermore, $$g(A)\odot g(B)=tr(gAg^{-1}gBg^{-1})=tr(gABg^{-1}=A\odot B.$$ 

 It is well know that this is a faithful, irreducible, representation. That is, we may think $G$ is a subgroup of $SO_5$  with the property that $G$ does not leave invariant any proper subspace.  
Moreover, it is well known that any other irreducible representation on $\mathbb R^5$ factors through $G$. 

 \medskip
 
 The following lemma finally proves that \emph{if $B$ is an $6$-dimensional  convex body with the property that all its hyperplane sections through the origin  are affinely equivalent,  then every hyperplane section of $B$ through the origin is an affine body of revolution.}

\medskip

\begin{lema}\label{lemG}
Let $\Omega \subset SO_5$ be a subgroup isomorphic to $SO_3$, Then, the structural group of $T\mathbb{S}^{5}$ can not be reduced to $\Omega$.
\end{lema}

\pf Suppose there is $f:\mathbb{S}^{4}\to \Omega$ such that $i_\Omega f:\mathbb{S}^{4}\to SO_5$ is homotopic to the characteristic map $\chi_{5}:\mathbb{S}^{4}\to SO_5$ of $T\mathbb{S}^{5}$, where $i_\Omega:\Omega\to SO_5$ is the inclusion.  Let $\pi:\mathbb{S}^{3}\to \Omega$ be the double covering map and let $g:\mathbb{S}^{3}\to \Omega$ be such that $\pi g=f$. 

Let $u:SU_2\to SO_5$ be the inclusion. Hence $\pi_3(SO_5)=\mathbb{Z}$ (every compact, simple Lie group has $\pi_3 = \mathbb{Z})$ and $u_*:\pi_3(SU_2)\to \pi_3(SO_5)$ is a isomorphism. On the other hand, at the level of homology, 
$H_3(SO_5,\mathbb{Z}_2)$ is a directed sum of $\mathbb{Z}_2$'s and $u_*:H_3(SU_2,\mathbb{Z}_2) \to H_3(SO_5,\mathbb{Z}_2)$ is not zero.
Let us consider $[i_\Omega \pi]\in \pi_3(SO_5)=\mathbb{Z}$. Suppose $[i_\Omega \pi]= m\in\mathbb{Z}$ and let  $\zeta:\mathbb{S}^3\to SU_2$ such that the induced homomorphism in 
homotopy is $\zeta_*(1)=m\in \pi_3(SU_2)=\mathbb{Z}$.
Consequently $u\zeta:\mathbb{S}^{3}\to SO_5$ is homotopic to $i_\Omega \pi:\mathbb{S}^{3}\to SO_5$. In $3$-dimensional homology,  $(i_G \pi)_*(1)=0$ which implies that $(u\zeta)_*(1)= 0$ and therefore, 
 since $u_*:H_3(SU_2,\mathbb{Z}_2) \to H_3(SO_5,\mathbb{Z}_2)$ is not zero, that $m$ is even.

Since $m$ is even, the map $\zeta g:\mathbb{S}^{4}\to SU_2$ is null homotopic, because $\zeta_*:\pi_4(\mathbb S^3)\to \pi_4(SU_2)$ is zero.  This is a contradiction to the fact $\mathbb{S}^{5}$ is not parallelizable. 
\qed

The intuitive claim: $u_*:H_3(SU_2,\mathbb{Z}_2) \to H_3(SO_5,\mathbb{Z}_2)$  is not zero, used in the above proof, is not so easy to prove. Indeed,  to justify it, it is necessary to use the Dynkin index.

  \bigskip
  
\subsection{Affine bodies of revolution}

A convex body $K\subset\R^{n}$ is a {\em body of revolution} if it admits an {\em axis of revolution}, i.e.,    a 1-dimensional  line $L$ such that  each  section of $K$ by an affine hyperplane $\Delta$ orthogonal to $L$ is  an $(n-1)$-dimensional euclidean ball in $\Delta$, centered at $\Delta\cap L$  (possibly empty or just a point). If $L$ is an axis of revolution of $K$ then $L^\perp$, is the  associated  {\em hyperplane of revolution}.  Clearly, a ball is a body of revolution and any line through its center serves as an axis of revolution.

An axis of revolution of a plane convex figure is an axis of  symmetry (or reflexion).  Of course, a convex figure may have two different axes of symmetry without being a disk.  In dimension $n\geq3$,  the situation is different. 

\begin{teo}\label{thmdos}
A convex body of revolution $K\subset \R^n$, $n\geq3$, with two different axes of revolution must be a ball.  
\end{teo}

\pf  Consider $G_K=\{g\in SO_n\mid g(K)=K\}$ be the collection of orientation preserving isometries that fixes $K$ and suppose $L\not=L'$ are two different axis of revolution of $K$. Without loss of generality, we may assume $L$ is the $1$-dimensional subspace orthogonal to $\R^{n-1}$. Clearly, the collection of orientation preserving isometries of $\R^n$ that fixes $L$ also fixes $K$ and  is equal to $SO_{n-1}\subset SO_n$. On the other hand, the group of  orientation preserving isometries of $\R^n$ that fixes $L'$ fixes also $K$ and  is equal to $SO'(n-1)$, a conjugate subgroup of $SO_{n-1}$ in $SO_n$. Thus, our hypotheses imply that $SO_{n-1}\subsetneq G_K\subset SO_n$, but it is well known that $SO_{n-1}$ is a {\it maximal connected} subgroup of $SO_n$ (see Lemma 4 of \cite{MS}, p.463). 
 Therefore, $G_K=SO_n$ and $K$ must be a ball.
\qed

An {\em affine  body of revolution} is a convex body affinely  equivalent to a  body of revolution. The images, under an affine equivalence, of an  axis of revolution and its  associated hyperplane of revolution of the body of revolution are  an axis of revolution  and associated hyperplane of revolution of the affine body of revolution (not necessarily perpendicular anymore). Clearly, an ellipsoid centered at the origin is an affine  body of revolution and any hyperplane through the origin serves as a hyperplane of revolution.

As in the euclidean case, an non-elliptical body of revolution admits a unique axis of revolution and a unique hyperplane of revolution.

\begin{cor}\label{lemapdos}
An affine convex body of revolution $K\subset \R^n$, $n\geq3$, with two different hyperplanes  of revolution must be an ellipsoid.  
\end{cor}

\pf Let $E$ be the ellipsoid of minimal volume containing $K$. By translation and dilatation of the principal axes of this ellipsoid, we obtain an affine isomorphism  $f:\R^{n}\to \R^{n}$ such that $f(E)$ is the unit ball of $\R^{n}$. Then since every affine isomorphism that fixes $f(K)$, also fixes $f(E)$, we have that  $f(K)$ contains two different axes of revolution. By Lemma \ref{thmdos}, $f(K)$ is a ball and consequently $K$ is an ellipsoid. 
\qed 

\subsubsection{Sections of affine bodies of revolution}

It is not difficult to see that every section of a body of revolution is a body of revolution, that is why sections of affine bodies of revolution are affine bodies of revolution. Is the converse true?  As far as I know, nobody knows the answer.

\begin{conj}
Suppose $B$ is an $(n+1)$-dimensional convex body all whose hyperplane sections through the origin are affine bodies of revolution, $n\geq3$. Then $B$ is  an affine body of revolution.
 \end{conj}
 
 We shall give a partial answer to this conjecture which will turn out to be sufficiently good for our purposes. Under the same hypothesis we shall prove that at least one section of $B$ through the origin is an ellipsoid.  If this is so, and if in addition, $B$ satisfies the 
 hypothesis that every two or its hyperplane section through the origin are affinely equivalent, then every section of $B$ through the origin is an ellipsoid and consequently $B$ is an ellipsoid. The proof of the existence of at least one elliptical section 
is a very interesting proof that combines ideas of convex geometry and algebraic topology. Before exposing it here, we require three intuitive lemmas, which we will state without proof. We ask the reader to include their own proofs.

 \begin{lema}\label{lemuabset} 
Every  hyperplane section $\Gamma\cap K$ of an affine body of revolution $K\subset \mathbb{R}^{n}$, $n\geq 3$,  is an affine  body of revolution. Furthermore, if $H$ is the  hyperplane of revolution of $K$, then either $\Gamma$ is parallel to $H$ or $\Gamma\cap H$ is a hyperplane of revolution of $\Gamma \cap K$.  
\end{lema}

\smallskip

\begin{lema}\label{lemaun1}
Let $K\subset  \R^{n}$, $n\geq 3$, be an affine body of revolution with axis of revolution the line $L$ and 
let $\Gamma$ be a hyperplane containing $L$.  Suppose that $\Gamma \cap K$ is an ellipsoid, then 
Then $K$ is an ellipsoid.
\end{lema}

\smallskip

\begin{lema}\label{lema:cont}
 Let $B\subset\R^{n+1}$ be a symmetric convex body, all of whose hyperplane sections are non-elliptical  affine symmetric bodies of revolution. For each $x\in \mathbb S^n$ let $L_x$ be the (unique) axis of revolution of $x^\perp\cap B$, where $x^\perp$ denotes the subspace orthogonal to $x$.
 Then $x\mapsto L_x$ is a continuous function $\mathbb S^n\to \R P^n$. Consequently,
 $$\{x+ L_x\}_{x\in \mathbb{S}^{n}}, $$
is a field of lines tangent to $\mathbb{S}^{n}.$
\end{lema}

\smallskip

Since every field of tangent lines give rise to a trivial 1-dimensional fiber bundle over $\mathbb{S}^{n}$, then there is
$$\psi:\mathbb{S}^{n}\to \mathbb{S}^{n},$$
such that for every $x\in \mathbb{S}^{n}$
$$L_x\cap \mathbb{S}^{n}=\{-\psi(x), \psi(x)\}.$$

Note that,  for every $x\in \mathbb{S}^{n}$, $\psi(x)$ is orthogonal to $x$ and hence $\psi(x)\not=-x$.  This implies that 
$\psi:\mathbb{S}^{n}\to \mathbb{S}^{n}$ is homotopic to the identity map and therefore that $\psi$ is suprayective. 

\bigskip

From now on, let $B\subset\R^{n+1}$ be a symmetric convex body, all of whose hyperplane sections are non-elliptical  affine symmetric bodies of revolution and remember that 
for every $u\in \mathbb{S}^{n-1}$, we denote by $u^\perp$ the $(n-1)$-dimensional subspace of $\mathbb{R}^{n}$ orthogonal to $u$.  Furthermore,  by Lemma \ref{lemapdos}, denote by $L_u$ the unique affine axis of revolution of $u^\perp\cap B$, and by 
$H_u$ the corresponding $(n-1)$-dimensional hyperplane of revolution of $u^\perp$.  
Note that the line $L_u$ contains the origin. The fact that $u^\perp\cap B$ is symmetric implies that the origin is the center of the ellipsoid  $H_u\cap B$ and therefore that the origin lies in $L_u$.

\begin{lema}\label{lemU}
Let $B\subset \mathbb{R}^{n+1}$ be a  symmetric convex body with center at origin, $n\geq 4$ and suppose that every hyperplane section of $B$ through the origin  is a non-elliptical  affine convex body of revolution. 
Suppose $L_v\subset u^\perp$ for some $u,v \in \mathbb{S}^{n-1}$. Then 
$$H_ v\cap \Gamma_u=H_u\cap \Gamma_v= H_v\cap H_u.$$ 
\end{lema}

\pf 
Consider  $u^\perp\cap v^\perp$, the $(n-1)$-dimensional subspace of $v^\perp$.  By hypothesis, $v^\perp\cap B$ is a non-elliptical  affine body of revolution with affine axis of revolution $L_v$. Therefore,  
since $L_v\subset u^\perp\cap v^\perp$, we have that $u^\perp\cap v^\perp\cap B$ is an affine body of revolution with affine axis of revolution $L_v$. Furthermore, by Lemma \ref{lemaun1}, $u^\perp\cap v^\perp\cap B$ is not an ellipsoid.  Moreover,  the principal affine subspace of revolution of 
$u^\perp\cap v^\perp\cap B$ is $H_v\cap u^\perp$. 

On the other hand, $u^\perp\cap v^\perp$, is an $(n-1)$-dimensional subspace of $u^\perp$.  Note that $u^\perp\cap v^\perp\not=H_u$, otherwise 
$u^\perp\cap v^\perp\cap B=H_u\cap B$, 
would be an ellipsoid, contradicting  our previous assumption.  Since $u^\perp\cap B$ is a non-elliptical  affine body of revolution and 
$u^\perp\cap v^\perp\not=H_u$, then by Lemma \ref{lemuabset}, $u^\perp\cap v^\perp\cap B$ is an affine body of revolution with principal affine subspace of revolution $H_u\cap v^\perp$. Consequently, by Lemma \ref{lemapdos}, we have that $H_ v\cap u^\perp=H_u\cap v^\perp$.
\qed

\smallskip

Our main results regarding affine bodies of revolution is:

\begin{teo}\label{thmAn}
Let $B\subset \mathbb{R}^{n+1}$ be a symmetric convex body with center at origin, $n\geq 4$ and suppose that every hyperplane section of $B$ through the origin  is an affine body of revolution. Then there is a hyperplane section through the origin of  $B$ which is 
an ellipsoid.
\end{teo}

\begin{proof} 
Suppose not, suppose $B$ is a symmetric convex body with center at the origin and with the property that every hyperplane section of $B$ through the origin  is a non-elliptical  affine convex body of revolution. 

Let us fix a point $x_0\in H_{u_0}\cap \mathbb{S}^{n}.$  Since $\psi:\mathbb{S}^{n}\to \mathbb{S}^{n}$ is suprayective,
let $v_0\in \mathbb{S}^{n}$ such that $\psi(v_0)=x_0$.
This implies that 
$$L_{v_0}\subset H_{u_0}.$$

This is a contradiction to  Lemma \ref{lemU}  because clearly $L_{v_0}\subset {u_0}^\perp,$ hence, 
$$ L_{v_0}\subset H_ {u_0}\cap {v_0}^\perp=H_{v_0}\cap {u_0}^\perp\subset H_{v_0},$$
which is impossible. 
\end{proof}

Theorem \ref{thmAn} is also true when $n=2$.  Indeed, in \cite{M2} Montejano proved that if $B$ is a $3$-dimensional convex body which contains the origin as interior point and every section through the origin is a figure that has a line of reflection (symmetry), then there is a section through the origin that is a disk. 
The proof also uses topology but it is intrinsically different to the proof of  Theorem \ref{thmAn}.  The case n = 3 remains open.

\bigskip 

With this, we have finished exposing the solution to the Banach conjecture over the reals given by Gromov \cite{Gv}, when $n=$ even and by
Bor-Hernandez-Jimenez-Montejano \cite{BHJM} when $n\equiv 1$ mod $4$, $n\not=133$. 
We summarize the results below in the following Theorem.

\begin{teo}[Main Theorem]\label{teomain}
If all hyperplane sections through the origin of an $(n+1)$-dimensional convex body $B$ are affinely equivalent, $n\equiv 0,1,2$ mod $4$, $n\not=133$,
 then the convex body $B$ is an ellipsoid.
\end{teo}

\bigskip

\subsection{The Banach Conjecture,  when $n$ is odd and dim $V\geq n+2$}

As we have mentioned before, the cases of the Banach conjecture that have yet to be solved are those in which $ n \equiv 3 $ mod $ 4 $. That is, the first unsolved case from the Banach conjecture is the following.

\begin{conj}\label{conj}
If all hyperplane sections through the origin of a $4$-dimensional convex body $B$ are affinely equivalent, 
 then the convex body $B$ is an ellipsoid.
\end{conj}

Indeed, Gromov in his original paper \cite{Gv}, using topology but a complete different sort of ideas,  proved the Banach conjecture over the reals,  
when $ n \equiv 3 $ mod $ 4 $ and dim $V>n+1$ and the Banach conjecture over the complex numbers, when $ n \equiv 3 $ mod $ 4 $ and dim $V>2n-1$.

The purpose of this section is to introduce these deep ideas. Let us prove the Banach conjecture over the reals, when  $n>1$ is odd and dim$V\geq n+2$.  

\smallskip

\begin{teo}[Gromov]\label{teo53}
Let $B$ be an $(n+2)$-dimensional symmetric convex body with center at the origin and suppose 
all $n$-sections through the origin  are linearly  equivalent, $n>1$ odd.  
 Then the convex body $B$ is an ellipsoid.
\end{teo}

Denote by $V_{n,k}$ the space of all orthonormal $k$-frames $(e_1,\dots,e_k)$, where $e_i\in \mathbb R^n$, $n\geq k.$
For our purpose, consider the space of $4$-frames $(e_1,e_2,e_3,e_4)$ in $\mathbb R^{n+2}$ and also the following two fiber bundles 
\medskip
 
 $$p_1: V_{n+2,4}\to V_{n+2,2},$$
$$p_2: V_{n+2,4}\to V_{n+2,2},$$
where
$$p_1(e_1,e_2,e_3,e_4)=(e_1,e_2),$$
and 
$$p_2(e_1,e_2,e_3,e_4)=(e_3,e_4).$$

\medskip

  The fiber in both cases  is   \emph{the Stiefel Manifold} $V_{n,2}$.  For more about Stiefel fiber bundles, see the book \cite{MiS}.
  
  \smallskip
  
  Consider now a nonempty closed subset $V\subset V_{n+2,2}$ and denote by: 
  $$\tilde V=p_1^{-1}(V)= \{((e_1,e_2,e_3,e_4)\in V_{n+2,4}\mid (e_1,e_2)\in V\}.$$
  
  \smallskip
  
  The following lemma is Proposition 3 of Gromov's paper \cite{G}.

 \begin{lema}\label{lempro}
 If $n$ is odd and the restriction $p_2|:\tilde V\to V_{n+2,2}$ is a fiber bundle, then $V=V_{n+2,2}$.
\end{lema}

We give only a brief sketch of the main ideas of the proof. We must consider an arbitrary fiber $V_{n,2}$ of $p_2$ and prove that the intersection 
$V'=\tilde V \cap V_{n,2}$ coincide with $V_{n,2}$. Note that the dimension of $V_{n,2}=2n-3$. In fact, if $V'\not=V_{n,2}$, then $H^{2n-3}(V;\mathbb Q)=0$ and for $p+q=2n-3$, the second term 
$E^{p,q}_2$ in the spectral sequence of the fiber bundle $p_2|:\tilde V\to V_{n+2,2}$ is trivial, which implies that $H^{2n-3}(\tilde V;\mathbb Q)$ is trivial, contradicting an old result of Borel in \cite{Bo}, p. 192, that claim that for $n=$ odd, the homomorphism induced by the inclusion 
$$H^{2n-3}(V_{n+4,2};\mathbb Q)\to H^{2n-3}(V_{n,2};\mathbb Q)$$
is non zero.
\bigskip

\noindent \emph{Proof of Theorem \ref{teo53}}. By hypothesis there is a symmetric convex body $K\subset \mathbb R^n$ with the property that the ellipsoid of minimal volume containing $K$ is the unit ball of $\mathbb R^n$ and such that every $n$-dimensional section of $B$ through the origin  is linearly equivalent to $K$.  Let us denote, as usual, by $G_K$  the Lie group of all linear isomorphisms of  $\mathbb R^n$ that keep $K$ fixed. Of course,  $G_K\subset O_n$.

Let us fix a $2$-dimensional plane $\Delta$ through the origin and define 
$V\subset V_{n+2,2}$ be the set of $2$-frames $(e_1,e_2)$ in $\mathbb R^{n+2}$ such that if $<e_1,e_2>$ is the subspace spanned by $e_1$ and $e_2$, then the section
$<e_1,e_2>\cap B$ is linearly equivalent to the section $\Delta \cap B.$
Furthermore, let $V'\subset V_{n,2}$ be the set of $2$-frames $(e_1,e_2)$ in $\mathbb R^n$ such that 
$<e_1,e_2>\cap K$ is linearly equivalent to $\Delta \cap B.$ Finally, let $\tilde V=p_1^{-1}(V)= \{((e_1,e_2,e_3,e_4)\in V_{n+2,4}\mid (e_1,e_2)\in V\}.$

We shall first prove that the restriction $p_2|:\tilde V\to V_{n+2,2}$ is a locally trivial bundle with fiber $V'$. 
For that purpose, consider $U$ an open contractible subset of $V_{n+2,2}$.  Then, using the contractibility of $U$ and the existence of a field of convex bodies lineally equivalent to $K$, contained in the fibers of the canonical vector bundle of $n$-subspaces in $\mathbb R^{n+2}$, 
it is possible to construct a continuous map $\Lambda:U\to GL(n,n+2)$ satisfying the following properties: 
\begin{enumerate}
\item For every $(e_3,e_4)\in U$, $\Lambda_{e_3,e_4}:\mathbb R^n\to \mathbb R^{n+2}$ is a linear embedding,
\item For every $(e_3,e_4)\in U$, $\Lambda_{e_3,e_4}(\mathbb R^n)$ is orthogonal to both $e_3$ and $e_4$,
\item For every $(e_3,e_4)\in U$, $\Lambda_{e_3,e_4}(K)= \Lambda_{e_3,e_4}(\mathbb R^n)\cap B.$
\end{enumerate}

Given a pair of linearly independent vector $(w_1,w_2)$, denote by $\big(GS^1(w_1,w_2), GS^2(w_1,w_2)\big)$ the $2$-frame obtained from  
$(w_1,w_2)$ by the Gram-Schmidt procedure in such way that $<w_1,w_2>=<GS^1(w_1,w_2), GS^2(w_1,w_2)>.$

Define the following fiber preserving map: 
$$\Phi:U\times V' \to V_{n+2,4},$$
given by  $\Phi\big((e_3,e_4), (e_1,e_2)\big)= 
\big(GS^1(\Lambda_{e_3,e_4}(e_1),\Lambda_{e_3,e_4}(e_2)), GS^2(\Lambda_{e_3,e_4}(e_1),\Lambda_{e_3,e_4}(e_2)), e_3, e_4\big)$.

First of all, by (2), $\big(GS^1(\Lambda_{e_3,e_4}(e_1),\Lambda_{e_3,e_4}(e_2)), GS^2(\Lambda_{e_3,e_4}(e_1),\Lambda_{e_3,e_4}(e_2)), 
e_3, e_4\big)\in V_{n+2,4}$. 
Moreover, by (1), $\big(GS^1(\Lambda_{e_3,e_4}(e_1),\Lambda_{e_3,e_4}(e_2)), GS^2(\Lambda_{e_3,e_4}(e_1),\Lambda_{e_3,e_4}(e_2))\big)\in V$ and therefore 
$ \big(GS^1(\Lambda_{e_3,e_4}(e_1),\Lambda_{e_3,e_4}(e_2)), GS^2(\Lambda_{e_3,e_4}(e_1),\Lambda_{e_3,e_4}(e_2)), 
e_3, e_4\big)\in  p_2|^{-1}(U)$. Hence,
we obtain a fiber preserving homeomorphism 

$$
\begin{tikzcd}[row sep=large]
U\times V'\arrow[swap,"proj",d] \arrow[r, hook,"\Phi"] & p_2|^{-1}(U)\arrow[d, "p_2|"]\\
U \arrow[swap,r,"id"] &U
\end{tikzcd}
$$
thus proving that $p_2|:\tilde V\to V_{n+2,2}$ is a locally trivial bundle with fiber $V'$. Furthermore, $p_2$  is a fiber bundle with 
structure  group $G_K$.  If this is so, by Lemma \ref{lempro}, $V=V_{n+2,2}$.  This implies that for every two planes through the origin, the corresponding sections of $B$ are linearly equivalent and hence that $B$ is an ellipsoid. \qed

\bigskip

 \subsection{The complex Banach conjecture}

The fifth meaning of equal is complex affinely equivalence. 

\smallskip

\noindent \emph{Let $V$ be a finite dimensional Banach space over the complex numbers all of whose hyperplane subspaces are isometric to each other. Is it true that $V$ is a Hilbert space?} 

\smallskip

Our next purpose is to prove that the above problem is equivalent to the following geometric problem. We need first some definitions. 

\smallskip

Let $\mathbb{S}^1$ be the space of all unit complex numbers $\mathbb{C}$. 
Let $A$ be a subset of complex space $\mathbb{C}^n$.  We say that $A$ is \emph{complex symmetric} if and only if there is a translated copy $A'$ of $A$ such that 
$\xi A'=A'$, for every $\xi\in \mathbb{S}^1$. In this case, if  $A'=A-x_0$, we say that $x_0$ is the center of complex symmetry of $A$.
If $-A$ is a translated copy of $A$, then we just say that $A$ is \emph{symmetric}.    It will be useful to consider the empty set as a complex symmetric set.  Note that a \emph{ compact  convex set $A\subset \mathbb{C}^n$ is complex symmetric with center at $x_0$  if and only if for every complex line $L$ through $x_0$, the section $L\cap A$ is a disk centered at $x_0$.}  Of course, any complex $k$-plane or  a ball in a finite dimensional Banach space over the complex numbers  is complex symmetric. 
A complex ellipsoid is the image of a ball under a complex affine transformation. Thus, balls of finite dimensional Hilbert spaces are complex ellipsoids.  Of course, complex ellipsoids are complex symmetric spaces.  With this definitions  in mind, we may state the following problem equivalent to the complex Banach conjecture 

\smallskip

\noindent \emph{If all complex hyperplane sections through the origin of a convex body $B\subset \mathbb{C}^{n+1}$ with the origin as center of complex symmetry  are complex linearly equivalent, 
is the convex body $B$  a complex ellipsoid.}

  \smallskip
  
  As was already informed, this problem has a positive answer when $n=$ even (Gromov \cite{G}) and when $n\equiv 1$ mod $4$ (Bracho and Montejano \cite{BM}). The purpose of this section is to give a brief summary of the ideas and techniques used in the proof.
 
 \smallskip
 
 This time, unlike the real case in which we use the principal bundle $SO_n\hookrightarrow SO_{n+1}\to \mathbb{S}^{n}$, we will use the corresponding principal bundle  $SU_n\hookrightarrow SU_{n+1}\to \mathbb{S}^{2n+1}$. Here $SU_n$ is the group of complex isometries of determinant $1$ in $\mathbb C^n$ and we say that the structure group of the principal bundle $SU_n\hookrightarrow SU_{n+1}\to \mathbb{S}^{2n+1}$ can be reduced to $G\subset SU_n$ if 
 the characteristic map $\chi_n :\mathbb S^{2n}\to SU_n$ of the complex bundle factorize through  $G$, that is, there is a map 
 $f:\mathbb S^{2n}\to G$ such
  the following diagram commutes up to homotopy, where $i:G\to SU_n$ is the inclusion
  \medskip
$$
\begin{tikzcd}[column sep=small]
\mathbb S^{2n}\arrow[dr,"f",swap] \arrow[rr, "\chi_n"] & &SU_n \\
&G.\arrow[ur,"i",swap]&
\end{tikzcd}
$$

 \smallskip

 Denote by $GL'_n(\mathbb{C})$ the group of complex linear isomorphisms of $\mathbb{C}^{n}$ with determinant a positive real number. Note that if $K_1$ and $K_2$ are 
complex symmetric convex bodies in  $\mathbb{C}^n$ which are  complex linearly equivalent, then there is $g\in GL'_n(\mathbb{C})$ such that $g(K_1)=K_2$. 

Given a complex symmetric convex body  $K\subset \mathbb{C}^n$, let $G_K:=\{g\in GL'_n(\mathbb{C})| g(K)=K\}$ be the  {\em group of complex linear isomorphisms of $K$ with positive real determinant}.   By Lemma~1 of Gromow \cite{G} there exists a complex ellipsoid of minimal volume containing $K$  centered at the origin. Suppose now that this minimal ellipsoid is the $(2n-1)$-dimensional unit ball, then every  $g\in G_K$ is actually an element of $SU_n$, because it fixes the unit ball, so in this case, $G_K:=\{g\in SU_n\mid g(K)=K\}$.

The link between our geometric problem and the  topology  is via the following lemma:
\begin{lema}\label{lema:Ckey} Let $B\subset \mathbb{C}^{n+1}$, $n\geq2$, be a complex symmetric convex body with center at the origin all of whose complex hyperplane sections through the origin are  complex linearly equivalent. Then there exists a complex symmetric convex body $K\subset \mathbb{C}^n$ with center at the origin and with the property that every complex hyperplane section of $B$ is  complex linearly equivalent to $K$ and such that 
the structure group of  the principal fibre bundle $SU_n\hookrightarrow SU_{n+1}\to \mathbb{S}^{2n+1}$ can be reduced to  $G_K\subset SU_n$. 
\end{lema}

\smallskip 

Our main interest naturally lies in studying the structure groups of the principal bundle $\xi_n$: $SU_n\hookrightarrow SU_{n+1}\to \mathbb{S}^{2n+1}$. In particular, if $n\equiv 0$ mod $2$, $\xi_n$ cannot be reduced to a proper subgroup of $SU_{n-1}$ [Theorem 1B of Leonard \cite{L}]. Therefore, under the hypothesis of Lemma \ref{lema:Ckey}, $G_K$ must be $SU_n$, and hence $K$ must be a ball. This implies that every section of $B$ is a complex ellipsoid.  
Of course, every section of a complex symmetric body $B\subset \mathbb{C}^{n+1}$ is a complex ellipsoid only if $B$ is a complex ellipsoid. See
Lemma 3.3 of \cite{BM}. This proves the complex Banach conjecture, when $n$ is even. 
 
 \smallskip 
 
 For the case  $n\equiv 1$ mod $4$, the proof requires first studying the case in which $G_K\subset SU_n$ is irreducible. If so, the topology of compact Lie groups over the complex numbers is simpler than over the real numbers and then it is possible to prove, in a similar way to the real case, that $G_K=SU_n$. If this is the case, then every section of $B$ is an ellipsoid and consequently $B$ is also an ellipsoid.  If $G_K\subset SU_n$ is not irreducible but $G_K$ is a proper subgroup of $SU_n$, then we can prove that $G_K=SU_{n-1}$.  To understand the convex geometry of the consequences  of this result, we need the following definition:
 
 \medskip
 
A {\em complex body of revolution} is a complex symmetric convex body $K\subset \mathbb{C}^n$ for which there exists a $1$-dimensional complex subspace $L$ of $\mathbb{C}^n$, called its {\em axis of revolution}, such that for every  affine complex hyperplane $H$ orthogonal to $L$, we have that $H\cap K$ is either empty, a single point or a $(2n-2)$-dimensional ball centered at $H\cap L$.  Of course, $K$ is a convex body of revolution if and only if $G_K=SU_{n-1}$.

With this in mind, it is very clear that what we have obtained is the following theorem:

\begin{teo}\label{thmcCkey} 
Let $B\subset \mathbb{C}^{n+1}$, $n\equiv 1$ mod 4,  $n\geq 5$, be a complex symmetric convex body with center at the origin all of whose complex hyperplane sections through the origin are  complex linearly equivalent. Then, there exists a complex body of revolution $K\subset \mathbb{C}^n$ with center at the origin and with the property that every complex hyperplane section of $B$ through the origin is  $\mathbb{C}$-linearly equivalent to $K$.
\end{teo}

To conclude, we need to know what are the geometric consequences of all the complex hyperplane sections of a convex body being complex affine bodies  of revolution.

\begin{teo}\label{thm:rev}
A complex symmetric convex body $B\subset \mathbb{C}^{n+1}$ with center at the origin,  $n\geq 4$,  all of whose complex hyperplane sections 
through the origin are  complex affine bodies of revolution has at least one complex hyperplane section through the origin  which is a complex ellipsoid.    
\end{teo}

The proof of Theorem \ref{thm:rev} is similar to the proof of Theorem \ref{thmAn} except this time the proofs are just technically more complicated.
This concludes an sketch of the proof of the complex Banach Conjecture when $n\equiv 0, 1, 2$ mod $4$, because by Theorems \ref{thmcCkey}
and \ref{thm:rev}, every hyperplane section of $B$ through the origin is a complex ellipsoid and therefore, by Theorem 3.3 of \cite{BM} we obtain that $B$ is a complex ellipsoid as we wished. 

\smallskip

The following theorem follows immediately from Theorems \ref{thmcCkey} and \ref{thm:rev}. It proves the Banach conjecture over the complex numbers for $n\equiv 0,1,2$ mod $4$, and dim $V>n$.

\begin{teo}[Bracho-Montejano \cite{BM}] \label{teocmain}
If all complex hyperplane sections through the origin of a complex symmetric convex body $B\subset \mathbb C^{n+1}$ are linearly equivalent, $n\equiv 0,1,2$ mod $4$, 
 then the convex body $B$ is a complex ellipsoid.
\end{teo}

\section{ Convex bodies all whose orthogonal projections are equal}

 The purpose of this section is to answer the following question: 

\medskip

\emph{If all orthogonal projections of a convex body onto hyperplanes are "equal", 
is the convex body  "equal" to the ball? }

\smallskip

\subsection {Equal area, congruence and affine equivalence}

The first meaning of "equal" is same "area". In 1937, Aleksandrov A.D. \cite{A} proved that if all orthogonal projections of a symmetric convex body have the same area then not only does the body have the same volume of the corresponding ball but it is actually a ball.

\begin{teo}[Aleksandrov's projection theorem \cite{A}]\label{teopa}
If all orthogonal projections onto hyperplanes of a symmetric convex body $B\subset \mathbb R^{n+1}$ have equal 
$n$-dimensional volume, then  the convex body $B$ is a ball.
\end{teo}

  Without the hypothesis of symmetry, Theorem \ref{teopa} is false. But, 
a symmetric convex body all whose orthogonal projections have the same area, not only has the volume of the corresponding ball but it is actually a ball. 
For every $v\in \mathbb S^n$,  denote by  $B|v$  the orthogonal projection of $B$ onto $v^\perp$ and let $\nu(B|v)$ be the $n$-dimensional volume
of $B|v$. The proof of Theorem \ref{teopa} follows immediately from the following Aleksandrov's result (see Theorem 2.11.1 of \cite {MMO}). Given two convex bodies $B^1, B^2\subset \mathbb R^{n+1}$
symmetric with respect to the origin and such that $\nu(B^1|v)=\nu(B^2|v)$, for every $v\in \mathbb S^n$, then $B^1$ is a translated copy of $B^2$.  
The proof of this result is  analytic and a little more complicated than the proof of the Theorem  \ref{teoa}.

Using harmonic integration, it can be proved  that  a symmetric convex body all whose $(n-1)$-dimensional perimeter areas are equal must be a ball. The proof is similar to the proof of Theorem \ref{teoa} but using the support functions instead of the radial functions (see Theorem 4 of \cite{F}). Of course, without the symmetry hypothesis, the result is false as it can be observed with $3$-dimensional convex bodies of constant width 1,
 in which the perimeter of all their orthogonal projections are $\pi$. 
 
 \smallskip

The next meaning of "equal" is congruence. That is, assume that all orthogonal projections onto hyperplanes of the convex body $B\subset \mathbb R^{n+1}$ are congruent. 

The collection of orthogonal projections of $B\subset \mathbb R^{n+1},$
$$\{B|v\}_{v\in \mathbb S^n}$$ 
give rise, not only to a field of convex bodies congruent to $B|e_1$ and tangent to $\mathbb S^n$, but meanly to a complete turning of $B|e_1$,
where $e_1=\{1,0,\dots,\}\in \mathbb R^{n+1}$.  We know that a complete turning is only possible for symmetric convex bodies (see Section 3).  
So, $B|v$ is symmetric for every $v\in \mathbb S^n$ and consequently, it is not very difficult to prove that $B$ is symmetric, but in this last case
Aleksandrov's Theorem \ref{teopa} implies that $B$ is also a ball.  That is, we have the following theorem:

\begin{teo}\label{teopc}
If all orthogonal projections onto hyperplanes of a convex body $B\subset \mathbb R^{n+1}$ are congruent, then  the convex body $B$ is a ball.
\end{teo}

Suppose now all orthogonal projections onto hyperplanes of the convex body $B\subset \mathbb R^{n+1}$ are affinely equivalent to a convex body $K$  and  suppose without loss of generality  that the ellipsoid of minimal volume containing $K$ is the unit ball. Denote by $G_K:=\{g\in GL'_n(\mathbb{R})\mid g(K)=K\}\subset SO_n$. As in the case of the hyperplane sections, we have that the existence of the collection of projections $\{B|v\}_{v\in \mathbb S^n}$ give rise directly to the 
following lemma which is the link between the topology and the geometric problem.  Note that from the arguments given in the preceding paragraph and Theorem \ref{teom}, 
we may assume without loss of generality that $B$ and $K$ are symmetric with center at the origin.

\begin{lema}\label{lema:Pkey} Let $B\subset \mathbb{R}^{n+1}$, $n\geq2$, be a symmetric convex body all of whose orthogonal projections onto hyperplanes  are  linearly equivalent. Then there exists a symmetric convex body $K\subset \mathbb{C}^n$, with the property that every orthogonal projection of $B$ onto a hyperplane  is  linearly equivalent to $K$ and such that 
the structure group of  the principal fibre bundle $SO_n\hookrightarrow SO_{n+1}\to \mathbb{S}^{n}$ can be reduced to  $G_K\subset SO_n$. 
\end{lema}

Once we have this technical lemma, we are in a position to know, using the topological arguments from Section 4.2, how the projections of $B$ are. That is, we have the following theorem:

 \begin{teo}\label{thmcPkey} 
Let $B\subset \mathbb{R}^{n+1}$, $n\equiv 0,1, 2$ mod 4,  $n\geq 2$, be a convex body all of whose orthogonal projections onto hyperplanes are affinely equivalent. Then, there exists a body of revolution $K\subset \mathbb{R}^n$, with the property that every orthogonal projection of $B$ is  affinely equivalent to $K$.
\end{teo}

To conclude, we need to know the  geometric consequences of all orthogonal projections of a convex body being  affine bodies  of revolution.
Every orthogonal projection of a body of revolution is a body of revolution, this is why  projections of affine bodies of revolution are affine bodies of revolution. Is the converse true?  As far as I know, nobody knows the answer.  The following geometric question is of great interest.
\emph{Suppose $B$ is an $(n+1)$-dimensional convex body all whose orthogonal projections are affine bodies of revolution, $n\geq3$. Is $B$  an affine body of revolution?}
 
 We shall give a partial answer to this question which will turn out to be sufficiently good for our purposes. Under the same hypothesis of the above question we shall prove that at least one orthogonal projection of $B$ is an ellipsoid.  If this is so, and if in addition, $B$ satisfies the 
 hypothesis that every two or its orthogonal projections are affinely equivalent, then every orthogonal projection of $B$ is an ellipsoid and consequently $B$ is an ellipsoid. The proof of the next theorem is  very similar to the proof of Theorem \ref{thmAn}, with the different adjustments that are always necessary when trying to adapt a proof for sections to one for projections.

 \begin{teo}\label{thmPAn}
Let $B\subset \mathbb{R}^{n+1}$ be a symmetric convex body, $n\geq 4$ and suppose that every orthogonal projection onto hyperplanes of $K$ is an affine body of revolution. Then there is an orthogonal projection of  $B$ which is 
an ellipsoid.
\end{teo}

This result, together with Theorem \ref{thmcPkey}, immediately implies the following characterization of the ellipsoid first proved by Montejano in \cite{M3}.
 
\begin{teo}\label{thmP} 
Let $B\subset \mathbb{R}^{n+1}$, $n\equiv 0,1, 2$ mod 4,  $n\geq 2$, be a  convex body all of whose orthogonal projections onto hyperplanes are affinely equivalent. Then $B$ is an ellipsoid.
\end{teo}

\smallskip

\subsection{ The codimension 2 case for orthogonal projections}
In this section, we will adapt Gromov's ideas from Section 4.5 to the context of orthogonal projections.

\medskip

We need first a technical lemma.
\smallskip

\begin{lema}\label{lemAA}
Given a linear embeding $h:\mathbb R^n\to\mathbb R^m$, $2< n<m$, there is a continuous map $h^*:V_{n,2}\to V_{m,2}$ such that, 
for every $u\in  V_{n,2}$, i) $<h^*(u)>$ $\subset h(\mathbb R^n)$ and ii) $h(<u>^\perp)$ is orthogonal to $<h^*(u)>$, where $<u>$ denotes the plane generated by $u$.   

Furthermore, $h^*$ varies continuously with $h$, while $h$ varies in the space of linear embeddings from $R^n$ to $R^m$.
\end{lema}

\pf Let $H\subset h(\mathbb R^n)$ be the plane such that $H$ is orthogonal to $h(<u>^\perp)$ and let $\pi:h(\mathbb R^n)\to H$ be the orthogonal projection. Then, given $u=(u_1,u_2)\in V_{n,2}$, let $h^*(u_1,u_2)= \big(GS^1(\pi(u_1),\pi(u_2)), GS^2(\pi(u_1),\pi(u_2))\big)\in V_{m,2}$, where given a pair of linearly independent vector $(w_1,w_2)$, denote by $\big(GS^1(w_1,w_2), GS^2(w_1,w_2)\big)$  the $2$-frame obtained from  $(w_1,w_2)$ by the Gram-Schmidt procedure in such a way that  $<w_1,w_2> = $ $<GS^1(w_1,w_2), GS^2(w_1,w_2)>.$ \qed

\smallskip

Here is  the analogous of Theorem \ref{teo53} for orthogonal projections:
\smallskip

\begin{teo}[Montejano]\label{teop53}
Let $B$ be an $(n+2)$-dimensional symmetric convex body with center at the origin and suppose 
all orthogonal projections onto $n$-planes are linearly  equivalent, $n>1$  odd.
 Then the convex body $B$ is an ellipsoid.
\end{teo}

\pf There is a symmetric convex body $K\subset \mathbb R^n$ with the property that the minimal ellipsoid containing $K$ is the unit ball of $\mathbb R^n$ and such that every  orthogonal projections of B onto an $n$-dimensional subspace are linearly equivalent
to $K$.  
Let us fix a $2$-dimensional plane $\Delta$ through the origin and define 
$V\subset V_{n+2,2}$ be the set of $2$-frames $(e_1,e_2)$ in $\mathbb R^{n+2}$ such that the orthogonal projection of $B$ onto $<e_1,e_2>$ is linearly equivalent to the orthogonal projection of $B$ onto $\Delta$.
Furthermore, let $V'\subset V_{n,2}$ be the set of $2$-frames $(e_1,e_2)$ in $\mathbb R^n$ such that 
the orthogonal projection of $K$ onto 
$<e_1,e_2>$ is linearly equivalent to the orthogonal  projection of $B$ onto $\Delta$. Finally, let $\tilde V=p_1^{-1}(V)= \{((e_1,e_2,e_3,e_4)\in V_{n+2,4}\mid (e_1,e_2)\in V\}.$

We shall first prove that the restriction $p_2|:\tilde V\to V_{n+2,2}$ is a locally trivial bundle with fiber $V'$. 
For that purpose, consider $U$ an open contractible subset of $V_{n+2,2}$.  Then, using the contractibility of $U$ and the existence of a field of convex bodies, lineally equivalent to $K$, contained in the fibers of the canonical vector bundle of $n$-subspaces in $\mathbb R^{n+2}$, 
it is possible to construct a continuous map $\Lambda:U\to GL(n,n+2)$ satisfying the following properties: 

\smallskip

\begin{enumerate}
\item For every $(e_3,e_4)\in U$, $\Lambda_{e_3,e_4}:\mathbb R^n\to \mathbb R^{n+2}$ is an embedding,
\item For every $(e_3,e_4)\in U$, $\Lambda_{e_3,e_4}(\mathbb R^n)$ is orthogonal to both $e_3$ and $e_4$,
\item For every $(e_3,e_4)\in U$, $\Lambda_{e_3,e_4}(K)$ is the orthogonal projection of $B$ onto $\Lambda_{e_3,e_4}(\mathbb R^n)$.\end{enumerate}

\smallskip

Define the following fiber preserving map 

\smallskip
 
$$\Phi:U\times V' \to V_{n+2,4}$$
given by  $\Phi\big((e_3,e_4), (e_1,e_2)\big)= \big(h^*(e_1,e_2), e_3, e_4\big)$.

\smallskip

First of all, by (2), $ \big(h^*(e_1,e_2), e_3, e_4\big)\in V_{n+2,4}$. 
Moreover, by (1) and Lemma \ref{lemAA},  $\big(h^*(e_1,e_2)\big)\in V$ and therefore 
$ \big(h^*(e_1,e_2), e_3, e_4\big)\in p_2|^{-1}(U)$. Hence,
we obtain a fiber preserving homeomorphism 

\smallskip

$$
\begin{tikzcd}[row sep=large]
U\times V'\arrow[swap,"proj",d] \arrow[r, hook,"\Phi"] & p_2|^{-1}(U)\arrow[d, "p_2"]\\
U \arrow[swap,r,"id"] &U.
\end{tikzcd}
$$
Thus  proving that $p_2|:\tilde V\to V_{n+2,2}$ is a locally trivial bundle with fiber $V'$.  If this is so, by Lemma \ref{lempro}, $V=V_{n+2,2}$.  
This implies that every two orthogonal projections onto $2$-dimensional planes are affinely equivalent and hence, by Theorem \ref{thmP}, for $n=2$, that $K$ is an ellipsoid. \qed

\smallskip

\sn{\bf Acknowledgments.} We acknowledges  support  from CONACyT under 
project 166306 and  from PAPIIT-UNAM under project IN112614. We thank Rolf Schneider for his comments.

\end{document}